\documentclass{amsart}

\usepackage{amsmath, amssymb,amsthm}
\usepackage[a4paper]{geometry}

\newcommand{\norm}[1]{\lVert#1\rVert}
\newcommand{\R}{\mathbb{R}}
\newcommand{\Z}{\mathbb{Z}}
\newcommand{\N}{\mathbb{N}}

\newcommand{\Cay}{\mathrm{Cay}}

\newcommand{\modulus}[1]{\left| #1 \right|}
\newcommand{\Supp}[1]{\mathrm{Supp}(#1)}
\newcommand{\Magnus}[1]{\left(\begin{array}{cc} \theta(#1) & \frac{\partial^\star #1}{\partial x_1}t_1 + \ldots + \frac{\partial^\star #1}{\partial x_r}t_r \\ 0 & 1 \end{array}\right)}
\newcommand{\SolMX}[2]{{\left(\begin{array}{cc} #1 & #2 \\ 0 & 1 \end{array}\right)}}
\newcommand{\ddxi}[1]{\frac{\partial #1}{\partial x_i}}
\newcommand{\dstardxi}[1]{\frac{\partial^\star #1}{\partial x_i}}

\newcommand{\maggeo}{\varphi_{\mathrm{geo}}}
\newcommand{\Id}{\mathrm{Id}}

\newtheorem{thm}{Theorem}[section]{\bf}{\it}
\newtheorem{cor}[thm]{Corollary}{\bf}{\it}
\newtheorem{lem}[thm]{Lemma}{\bf}{\it}

\newtheorem{thmspecial}{Theorem}{\bf}{\it}

\title{Metric behaviour of the Magnus embedding}
\author{Andrew W. Sale}

\thanks{The author was supported by the EPSRC. The final publication is available at Springer via \texttt{http://dx.doi.org/10.1007/s10711-014-9969-z}}

\begin{document}

\begin{abstract}
The classic Magnus embedding is a very effective tool in the study of abelian extensions of a finitely generated group $G$, allowing us to see the extension as a subgroup of a wreath product of a free abelian group with $G$. In particular, the embedding has proved to be useful when studying free solvable groups. An equivalent geometric definition of the Magnus embedding is constructed and it is used to show that it is $2$-bi-Lipschitz, with respect to an obvious choice of generating sets. This is then applied to obtain a non-zero lower bound on $L_p$ compression exponents in free solvable groups.

 \medskip

\footnotesize{\noindent \textbf{2010 Mathematics Subject
Classification: 20F16, 20F65}  }
\end{abstract}

\maketitle

\section{Introduction}

The Magnus embedding is a valuable tool in the study of free solvable groups. If $N$ is a normal subgroup of a (non-abelian) free group $F$ of rank $r$, whose derived subgroup is denoted $N'$, then the Magnus embedding expresses $F/N'$ as a subgroup of the wreath product $\Z^r \wr F/N$. The embedding was introduced in 1939 by Wilhelm Magnus \cite{Magn39}, and in the 1950's Fox, with a series of papers \cite{Fox53}, \cite{Fox54}, \cite{Fox56}, \cite{CFL58}, \cite{Fox60}, developed a notion of calculus on free groups which enabled the Magnus embedding to be further exploited. We will give a geometric definition for the Magnus embedding, with which we show the following:

\begin{thmspecial}\label{thmspecial:magnus q.i.}
The Magnus embedding $\varphi:F/N' \hookrightarrow \Z^r \wr F/N$ is $2$--bi-Lipschitz for an appropriate choice of word metrics.
\end{thmspecial}

A consequence of this is that the Magnus embedding will be a quasi-isometric embedding for word metrics coming from any choice of finite generating sets.

In the realm of solvable groups, the Magnus embedding proves to be a very useful tool. In particular it provides a means to study \emph{free solvable groups}. Given a free group $F$ of rank $r$, denote by $F^{(d)}$ its $d$--th derived subgroup. The free solvable group of rank $r$ and derived length $d$ is the quotient
		$$S_{r,d} = F/F^{(d)}.$$
As its name suggests, the free solvable group $S_{r,d}$ is the free group in the variety of $r$--generated solvable groups of derived length $d$. Hence, any finitely generated solvable group is a quotient of a free solvable group. As is well known, the Magnus embedding enables us to see $S_{r,d+1}$ as a subgroup of $\Z^r \wr S_{r,d}$.

We apply Theorem \ref{thmspecial:magnus q.i.} to study the $L^p$ \emph{compression exponents} for free solvable groups. Compression exponents were first introduced by Guentner and Kaminker \cite{GK04}, building on the idea of uniform embeddings introduced by Gromov \cite{Grom93} and G. Yu \cite{Yu00}. In particular we use results of Naor and Peres \cite{NP11} and Li \cite{Li10} to show that free solvable groups have non-zero Hilbert compression exponent.

\begin{thmspecial}
For $r,d \in \N$, $r,d\neq 1$, the $L^p$ compression exponent for $S_{r,d}$ satisfies
				$$\frac{1}{d-1}\max\left\{\frac{1}{p},\frac{1}{2}\right\} \leq \alpha^\star_p(S_{r,d})\textrm{.}$$
\end{thmspecial}

Other applications of the Magnus embedding and Theorem \ref{thmspecial:magnus q.i.} are developed in \cite{Sale12freesolvableconjugacy}, where a quantitative version of Max Dehn's conjugacy problem is investigated in the setting of wreath products and free solvable groups.

\section{Preliminaries}\label{sec:preliminaries}

\subsection{Restricted Wreath Product}

Let $A,B$ be finitely generated groups. Denote by $A^{(B)}$ the set of all functions from $B$ to $A$ with finite support, and equip it with pointwise multiplication to make it a group. The (restricted) wreath product $A \wr B$ is the semidirect product $A^{(B)} \rtimes B$. To be more precise, the elements of $A\wr B$ are pairs $(f,b)$ where $f \in A^{(B)}$ and $b \in B$. Multiplication in $A\wr B$ is given by
				$$(f,b)(g,c)=(fg^b,bc), \ \ \ f,g \in A^{(B)} , \  b,c, \in B$$
where $g^b(x)=g(b^{-1}x)$ for each $x \in B$. The identity element in $B$ will be denoted by $\Id_B$, while we use $1$ to denote the trivial function from $B$ to $A$.

We can paint a picture of $A \wr B$ in a similar vein to the well-known picture for lamplighter groups $\Z_q \wr \Z$. In the more general context where we consider $A \wr B$, the problem of determining the length of an element requires a solution to the travelling salesman problem on a Cayley graph $\Cay(B,X)$ of $B$, with respect to some finite generating set $X$.  Suppose we take an element $(f,b) \in A \wr B$. We can think of this as a set of instructions given to a salesman, who starts the day at the vertex in $\Cay(B,X)$ labelled by the identity. The instructions comprise
\begin{itemize}
\item a list of vertices to visit (the support $\Supp{f}$);
\item a particular element of $A$ to ``sell'' at each of these vertices (determined by the image of $f$ at each vertex); and
\item a final vertex $b$, where the salesman should end the day.
\end{itemize}
Intuitively, therefore, we would expect the word length of $(f,b)$ to be the ``quickest'' way to do this. In particular, the salesman needs to find the shortest route from the identity vertex to $b$ in which every vertex of $\Supp{f}$ is visited at least once. We will denote the length of such a path by $K(\Supp{f},b)$, following the notation of \cite{deCo06}.
 
The following Lemma formalises this idea. A proof of the Lemma, for a slightly more general context, can be found in the Appendix of \cite[Lemma A.1]{deCo06} and also in \cite[Theorem 3.4]{DO11}. We fix a finite generating set $X$ for $B$ and for each $b \in B$ denote the corresponding word-length as $\modulus{b}$. We consider the left-invariant word metric on $B$, given by $d_B(x,y):=\modulus{x^{-1}y}$. Similarly, fix a finite generating set $T$ for $A$ and let $\modulus{\cdot}$ denote the word-length. For $f \in A^{(B)}$, let
				$$\modulus{f} = \sum_{x \in B} \modulus{f(x)}\textrm{.}$$
Let $A_{\Id_B}$ be the subgroup of $A^{(B)}$ consisting of those elements whose support is contained in $\{ \Id_B \}$. Then $A_{\Id_B}$ is generated by $\{f_t \mid t \in T \}$ where $f_t(\Id_B)=t$ for each $t \in T$ and $A \wr B$ is generated by $\{ (1,x) \mid x \in X \} \cup \{ (f_t,\Id_B) \mid t\in T \} $. With respect to this generating set, we will let $\modulus{(f,b)}$ denote the corresponding word-length for $(f,b) \in A \wr B$.

\begin{lem}[{\cite[Lemma A.1]{deCo06}}]\label{lem:wreath metric}
Let $(f,b) \in  A \wr B$, where $A,B$ are finitely generated groups. Then
$$
\modulus{(f,b)} = K(\Supp{f},b) + \modulus{f}$$
where $K(\Supp{f},b)$ is the length of the shortest path in the Cayley graph $\Cay(B,X)$ of $B$ from $\Id_B$ to $b$, travelling through every point in $\Supp{f}$.
\end{lem}

\subsection{Fox Calculus}
In order to define and make effective use of the Magnus embedding we need to understand Fox derivatives. These were introduced by Fox in the 1950's in a series of papers \cite{Fox53}, \cite{Fox54}, \cite{Fox56}, \cite{CFL58}, \cite{Fox60}.

Recall that a derivation on a group ring $\Z(G)$ is a mapping $\mathcal{D}:\Z(G)\rightarrow \Z(G)$ which satisfies the following two conditions for every $a,b \in \Z(G)$:
\begin{eqnarray*}
 \mathcal{D}(a+b)&=&\mathcal{D}(a)+\mathcal{D}(b) \\
 \mathcal{D}(ab)&=&\mathcal{D}(a)\varepsilon(b)+a\mathcal{D}(b)
\end{eqnarray*}
where $\varepsilon : \Z(G) \rightarrow \Z$ is the additive homomorphism that sends each element of $G$ to $1$.

Given a group homomorphism $\theta : G \to H$, we can naturally extend it to a ring homomorphism $\theta^\star : \Z(G) \to \Z(H)$. By composing a derivation $\mathcal{D} : \Z(G) \to \Z(G)$ with $\theta^\star$ we obtain a derivation $\mathcal{D}^\star : \Z(G) \to \Z(H)$, satisfying:
\begin{eqnarray*}
 \mathcal{D}^\star(a+b)&=&\mathcal{D}^\star(a)+\mathcal{D}^\star(b) \\
 \mathcal{D}^\star(ab)&=&\mathcal{D}^\star(a)\varepsilon(b)+\theta^\star(a)\mathcal{D}^\star(b).
\end{eqnarray*}

Suppose $G=F$, the free group on generators $X=\{x_1,\ldots , x_r\}$. For each generator we can define a unique derivation $\frac{\partial}{\partial x_i}$ which satisfies
				$$\frac{\partial x_j}{\partial x_i} = \delta_{ij} \Id_F$$
where $\delta_{ij}$ is the Kronecker delta and $\Id_F$ is the identity element of $F$. Any derivation $\mathcal{D}$ can be expressed as a unique $\Z(F)$--linear combination of these: for each $\mathcal{D}$ there exist unique elements $k_i \in \Z(F)$ such that
				$$\mathcal{D}(a)=\sum_{i=1}^{n} k_i \frac{\partial a}{\partial x_i}$$
for each $a \in \Z(F)$. 

Fox describes the following Lemma as the ``fundamental formula'' and it can be found in \cite[(2.3)]{Fox53}. The proof is straight-forward computation.

\begin{lem}[Fundamental formula of Fox calculus]\label{lem:fundamental theorem of fox calculus}
Let $a \in \Z(F)$. Then
				$$a-\varepsilon(a)\Id_F=\sum_{i=1}^{r} \ddxi{a} (x_i -1)\mathrm{.}$$
\end{lem}

Fox derivatives also accept a form of integration, see \cite[Ch.VII (2.10)]{CF63}. In particular, given $\beta_1, \ldots , \beta_r \in \Z(F)$ one can find $c \in \Z(F)$ such that $\frac{\partial c}{\partial x_i} = \beta_i$ for each $i$. The element $c$ is unique up to addition of scalar multiples of the identity.

Given a normal subgroup $N$ in $F$ and a derivation $\mathcal{D}$ of $\Z(F)$ we can consider the derivation $\mathcal{D}^\star:\Z(F) \rightarrow \Z(F/N)$, defined by the composition of the map $\mathcal{D}$ with $\theta^\star : \Z(F) \to \Z(F/N)$, the extension of the quotient homomorphism $\theta:F \rightarrow F/N$.

The following Lemma can be deduced from the Magnus embedding, but it also follows from \cite[(4.9)]{Fox53}.

\begin{lem}\label{lem:Derivation kernel in F}
Let $g \in F$. Then $\mathcal{D}^\star(g)=0$ for every derivation $\mathcal{D}$ if and only if $g \in N'=[N,N]$.
\end{lem}

\section{The Magnus embedding}\label{subsec:The Magnus Embedding}

\subsection{Definition via Fox calculus}
The Magnus embedding was first defined in \cite{Magn39}. The first definition we give here is the same as that given by Magnus, though we use the language of Fox derivatives.

Let $F$ be the free group of rank $r$ on the generators $X = \{ x_1 , \ldots , x_r \}$ and let $N$ be a normal subgroup of $F$. The Magnus embedding gives a way of recognising $F/N'$, where $N'$ is the derived subgroup of $N$, as a subgroup of the wreath product $M(F/N)=\Z^r \wr F/N$.

Consider the group ring $\Z(F/N)$ and let $\mathcal{R}$ be the free $\Z(F/N)$--module with generators $t_1, \ldots , t_r$. We define a homomorphism
				$$\varphi : F \longrightarrow M(F/N)=\left( \begin{array}{cc}F/N & \mathcal{R} \\ 0 & 1 \end{array}\right)=\left\{ \left(\begin{array}{cc} g & a \\ 0 & 1 \end{array}\right) \mid g \in F/N, a \in \mathcal{R} \right\}$$
by
				$$\varphi(w) = \Magnus{w}$$
where $\theta$ is the quotient homomorphism $\theta:F \rightarrow F/N$. Magnus \cite{Magn39} recognised that the kernel of $\varphi$ is equal to $N'$ and hence $\varphi$ induces an injective homomorphism from $F/N'$ to $M(F/N)$ which is known as the \emph{Magnus embedding}. In the rest of this paper we will use $\varphi$ to denote both the homomorphism defined above and the Magnus embedding it induces.

Given $w \in F$, its image under the Magnus embedding can be identified with $(f,b) \in \Z^r \wr F/N$ in the following way: we take $b=\theta(w)$ and $f$ will be the function $f^{(w)}=(f_1^{(w)}, \ldots, f_r^{(w)})$, where for each $i$ the function $f_i^{(w)} : F/N \rightarrow \Z$ satisfies the equation
				$$\sum_{g \in F/N} f_i^{(w)}(g)g = \dstardxi{w} \in \Z(F/N)\textrm{.}$$

Let $d_{F/N'}$ denote the word metric in $F/N'$ with respect to the generators determined by the image of elements of $X$ under the quotient map and let $d_M$ denote the word metric on $M(F/N)$ with respect to the generating set  
				$$\left\{ \SolMX{\theta(x_1)}{0} , \ldots , \SolMX{\theta(x_r)}{0} , \SolMX{1}{t_1} , \ldots , \SolMX{1}{t_r} \right\}\textrm{.}$$
Note that this generating set is the same as that used for Lemma \ref{lem:wreath metric}.
The aim is to compare the metrics $d_{F/N'}$ and $d_M$. We will do this by giving an equivalent definition of the Magnus embedding which is described by its geometric properties.

\subsection{Geometric definition} Define a multigraph $\Gamma$ as follows. Let the vertex set of $\Gamma$ be $F/N$. For each $g \in F/N$ and $x \in X$ connect $g$ to $gx$ by an edge labelled by $x$. Denote this edge by $(g,x)$. In many cases $\Gamma$ will be the Cayley graph $\Cay(F/N, \overline{X})$, where $\overline{X}$ is the image of $X$ in the quotient map. However, if, for example, there exist distinct $x,y \in X$ such that $xN=yN$ then $\Gamma$ will have two distinct edges from $g$ to $gx=gy$, for each $g \in F/N$, while $\Cay(F/N,\overline{X})$ will have just one.

Take a word $w$ in $F$ and construct the path $\rho_w$ read out by this word in the multigraph $\Gamma$. Let $E$ be the edge set of $\Gamma$. Define a function $\pi_w : E \rightarrow \Z$ such that for each edge $(g,x) \in E$ the value of $\pi_w(g,x)$ is equal to the net number of times the path $\rho_w$ traverses this edge --- for each time the path travels from $g$ to $gx$ along $(g,x)$ count $+1$; for each time the path goes backwards, from $gx$ to $g$, along $(g,x)$ count $-1$. Note that, since $F$ is free, $\pi_w=\pi_u$ whenever the words $w$ and $u$ represent the same element in $F$.

Given $w \in F$ we will use $\pi_w$ to define a function $P_w:F/N\to \Z^r$ in the natural way:
		$$P_w(g)=\big(\pi_w(g,x_1),\ldots , \pi_w(g,x_r)\big).$$
Define the \emph{geometric Magnus embedding} to be the function $\maggeo:F \to \Z^r \wr F/N$  such that $\maggeo(w)=(P_w,\theta(w))$ for $w \in F$.

\begin{thm}\label{thm:geometric magnus}
The two definitions of the Magnus embedding, $\varphi$ and $\maggeo$, are equivalent.
\end{thm}

\begin{proof}
To do this we need to show that for each $w \in F$ the maps $f^{(w)}$ and $P_w$ are in fact the same. In particular we need that $f_i^{(w)}(g) = \pi_w(g,x_i)$ for each edge $(g,x_i)$ in $\Gamma$.

We will prove this by induction on the word-length of $w$. If $w=x_j$ then $\dstardxi{w}=\delta_{ij}\Id_{F/N}$. The path $\rho_w$ consists of just one edge: $(\Id_{F/N},x_j)$. Hence $\pi_w(g,x_i)$ is zero everywhere except when $g=\Id_{F/N}$ and $i=j$, where it takes the value $1$. Thus, in this case, $f^{(w)}=P_w$. If $w=x_j^{-1}$ then $\dstardxi{w}=-\delta_{ij}x_j^{-1}$. The path $\rho_w$ this time consists of the edge $(x_j^{-1},x_j)$ and one can check that $P_w=f^{(w)}$ holds here too.

Now suppose $w$ has length at least $2$ and that the claim holds for all words shorter than $w$. Suppose also that $w$ is of the form $w=w'x_j^\varepsilon$ where $w'$ is a non-trivial word and $\varepsilon = \pm 1$. Then
				$$\dstardxi{(w'x_j^\varepsilon)} = \dstardxi{w'} + \theta(w') \dstardxi{x_j^\varepsilon}$$
and it follows that $f_i^{(w)}(g)=f_i^{(w')}(g)$ whenever $i \neq j$. When $i = j$ we get
				$$f_i^{(w)}(g)=\left\{ \begin{array}{ll} f_i^{(w')}(g) & \textrm{if $\varepsilon = 1$ and $g \neq \theta(w')$, or $\varepsilon = -1$ and $g \neq \theta(w)$,} \\ f_i^{(w')}(g)+1 & \textrm{if $\varepsilon = 1$ and $g=\theta(w')$,}\\ f_i^{(w')}(g)-1 & \textrm{if $\varepsilon=-1$ and $g=\theta(w)$.}\end{array}\right.$$
Meanwhile, $\rho_w$ is obtained from $\rho_{w'}$ by attaching one extra edge on to its final vertex, namely the edge $(w', x_i)$ is attached if $\varepsilon=1$ or $(w,x_i)$ if $\varepsilon = -1$. Hence \mbox{$\pi_w(g,x_i)=\pi_{w'}(g,x_i)$} whenever $i \neq j$ and when $i=j$ we get
				$$\pi_w(g,x_i)=\left\{ \begin{array}{ll} \pi_{w'}(g,x_i) & \textrm{if $\varepsilon = 1$ and $g \neq \theta(w')$, or $\varepsilon = -1$ and $g \neq \theta(w)$,} \\ \pi_{w'}(g,x_i)+1 & \textrm{if $\varepsilon = 1$ and $g=\theta(w')$,}\\ \pi_{w'}(g,x_i)-1 & \textrm{if $\varepsilon=-1$ and $g=\theta(w)$.}\end{array}\right.$$
Thus, applying the inductive hypothesis gives $f_i^{(w)}(g) = \pi_w(g,x_i)$ and the equality of $P_w$ and $f^{(w)}$ therefore holds for all words $w$.
\end{proof}

\subsection{Metric behaviour of the Magnus embedding}

We will use this geometric definition of the Magnus embedding to show that its image is undistorted in $M(F/N)$.

If $w$ is a geodesic word for $g \in F/N'$ then the length of $\rho_w$ is equal to $d_{F/N'}(\Id_{F/N'},g)$. We need to compare its length with the size of $\varphi(g)$ in the wreath product. We saw above how, if $\varphi(g)=(f^{(w)},\theta(w))$, then the function $f^{(w)}$ describes the route which $\rho_w$ takes, telling us the net number of times $\rho_w$ transverses each edge. From this we deduce a relationship between the size of $g$ and the size of $\varphi(g)$.

In order to compare $d_{F/N'}$ and $d_M$ we will use an expression for word-lengths in $F/N'$ given by Droms, Lewin and Servatius \cite[Theorem 2]{DLS93}. For this we will need the following notation: Let $\Supp{\pi_w}$ denote the subgraph of $\Gamma$ containing all edges $e$ such that $\pi_w(e) \neq 0$. Consider a new path $\sigma(\pi_w)$ which is a path travelling through every point in $\Supp{\pi_w}\cup\{\Id_{F/N}\}$ so that it minimises the number of edges not contained in $\Supp{\pi_w}$. Let $W(\pi_w)$ denote this number.

\begin{lem}[Droms--Lewin--Servatius \cite{DLS93}]\label{lem:F/N' word length - Droms-Lewin-Servatius}
Let $w$ be a word on generators $X$ which determines the element $g \in F/N'$. Then
				$$d_{F/N'}(\Id_{F/N'},g)=\sum_{e \in E}\modulus{\pi_w(e)} + 2W(\pi_w)\textrm{.}$$
\end{lem}

\begin{thm}\label{thm:F/N' undistorted in M(F/N)}
The subgroup $\varphi(F/N')$ is undistorted in $M(F/N)$. To be precise, for each $g \in F/N'$
				$$\frac{1}{2}d_{F/N'}(\Id_{F/N'},g) \leq d_M(\Id_M,\varphi(g))\leq 2d_{F/N'}(\Id_{F/N'},g)\textrm{.}$$
\end{thm}

\proof
The upper bound is immediate since each generator in $X$ is mapped under $\varphi$ to the product of two generators of $M(F/N)$.

Let $w$ be a geodesic word on $X\cup X^{-1}$ representing $g \in F/N'$. Suppose $\varphi(w)=(f^{(w)},\theta(w))$. From Lemma \ref{lem:wreath metric} the word-length in $M(F/N)$ is given by
				$$d_M(\Id_M,\varphi(w)) = K(\Supp{f^{(w)}},\theta(w)) + \sum_{y \in F/N} \norm{f^{(w)}(y)}$$
where $\norm{\cdot}$ is the $\ell_1$--norm on $\Z^r$. The expression of $f^{(w)}$ in terms of $\pi_w$ which comes from the equality of $P_w$ and $F^{(w)}$ given by in Theorem \ref{thm:geometric magnus} leads us to the equation
				\begin{equation}\label{eq:flow and function}\sum_{y \in F/N} \norm{f^{(w)}(y)} = \sum_{e \in E} \modulus{\pi_w(e)}.\end{equation}

Since an edge $e$ is in $\Supp{\pi_w}$  only if one of its ends is in $\Supp{f^{(w)}}$, we see that $\Supp{f^{(w)}}$ is contained in the subgraph $\Supp{\pi_w}$. Take a path $q$ starting at $\Id_{F/N}$ and travelling through every point in $\Supp{f^{(w)}}$, in particular we may take $q$ to be a path realising $K(\Supp{f^{(w)}},(w))$. Any edge in $\Supp{\pi_w}$ which is not in this path must have one vertex lying in the path $q$. Adding these edges to $q$ (along with the corresponding backtracking) gives a new path $q'$ passing though every point of $\Supp{\pi_w}\cup\{\Id_{F/N}\}$. Note that every edge in $q'$ that is not in $\Supp{\pi_w}$ was already in $q$. Hence the length of $q$ is bounded below by the size of $W(\pi_w)$. In particular $W(\pi_w) \leq K(\Supp{f^{(w)}},\theta(w))$ and hence, by equation \eqref{eq:flow and function} and Lemma \ref{lem:F/N' word length - Droms-Lewin-Servatius},
				$$\frac{1}{2}d_{F/N'}(\Id_{F/N'},g) \leq d_M(\Id_M,\varphi(g))$$
thus proving the result.
\qed

\section{Compression Exponents}\label{sec:compression}

We can use the fact that the Magnus embedding is a quasi-isometric embedding to obtain a lower bound for the $L^p$ compression exponent of free solvable groups. The $L^p$ compression exponent is a way of measuring how a group embeds into $L^p$.

Let $G$ be a finitely generated group with word metric denoted by $d_G$ and let $Y$ be a metric space with metric $d_Y$. A map $f : G \rightarrow Y$ is called a \emph{uniform embedding} if there are two functions $\rho_\pm : \R_{\geq 0} \rightarrow \R_{\geq 0}$ such that $\rho_-(r) \rightarrow \infty$ as $r \rightarrow \infty$ and 
				$$\rho_-(d_G(g_1,g_2))\leq d_Y(f(g_1),f(g_2)) \leq \rho_+(d_G(g_1,g_2))$$
for $g_1,g_2 \in G$.

One can define the \emph{$L^p$ compression exponent} for a finitely generated group $G$, denoted by $\alpha^\star_p(G)$, to be the supremum over all $\alpha \geq0$ such that there exists a Lipschitz map $f : G \rightarrow L^p$ satisfying
				$$Cd_G(g_1,g_2)^\alpha \leq \norm{f(g_1)-f(g_2)}$$
for any positive constant $C$. For $p=2$, the Hilbert compression exponent, which is denoted by $\alpha^\star(G)$, was introduced by Guentner and Kaminker \cite{GK04}.

Of particular interest to us is what happens to compression under taking a wreath product. The first estimate for compression exponents in wreath products was given by Arzhantseva, Guba and Sapir \cite{AGS06} where they show that the Hilbert compression exponent of $\Z \wr H$, where $H$ has super-polynomial growth, is bounded above by $1/2$. More recently Naor and Peres have given a lower bound for the compression of $A \wr B$ when $B$ is of polynomial growth \cite[Theorem 3.1]{NP11}. 

\begin{thm}[Naor--Peres \cite{NP11}]
Let $A,B$ be finitely generated groups such that $B$ has polynomial growth. Then, for $p \in [1,2]$,
				$$\alpha^\star_p(A\wr B) \geq \min \left\{ \frac{1}{p},\alpha^\star_p(A) \right\}\textrm{.}$$
\end{thm}

Li showed in particular that a positive compression exponent is preserved by taking wreath products \cite{Li10}.

\begin{thm}[Li \cite{Li10}]
Let $A,B$ be finitely generated groups. For $p \geq 1$ we have
			$$\alpha^\star_p(A \wr B) \geq \max \left\{\frac{1}{p},\frac{1}{2}\right\}\min\left\{\alpha^\star_1(A),\frac{\alpha^\star_1(B)}{1 + \alpha^\star_1(B)}\right\}\textrm{.}$$
\end{thm}

We can deduce from the result of Naor and Peres that the $L^1$ compression exponent for $\Z^r \wr \Z^r$ is equal to $1$. Hence the $L^1$ compression exponent for free metabelian groups, using Theorem \ref{thm:F/N' undistorted in M(F/N)}, is also equal to $1$. Then, with the result of Li, induction on the derived length gives us that $\alpha_1^\star(S_{r,d}) \geq \frac{1}{d-1}$. Finally, another application of Li's result gives us the following:

\begin{cor}\label{cor:L^p compression exponent lower bound}
Let $r,d \in \N$. Then
$$\alpha_1^\star(S_{r,d}) \geq \frac{1}{d-1}$$
and for $p >1$
$$\alpha_p^\star(S_{r,d}) \geq \frac{1}{d-1}\max\left\{\frac{1}{p},\frac{1}{2}\right\}\textrm{.}$$
\end{cor}

It would be interesting to determine an upper bound on $\alpha_p^\star(S_{r,d})$, in particular to check if it is ever strictly less than $\frac{1}{2}$ since no example of a solvable group with non-zero Hilbert compression exponent strictly less than $\frac{1}{2}$ is known. Austin \cite{Aust11} has constructed solvable groups with $L^p$ compression exponent equal to zero. His examples are modified versions of double wreath products of abelian groups. Therefore to find such an example it seems natural to look in the class of iterated wreath products of solvable groups and special families of their subgroups, such as the free solvable groups.

\medskip

\textbf{Acknowledgements.}
The author would like to thank Cornelia Dru\c{t}u for many valuable discussions on this paper. Alexander Olshanskii's comments on a draft copy were also very helpful, as were discussions with Romain Tessera. He would also like to thank David Hume for useful discussions on $L^p$ compression exponents.

\bibliography{bibliography}{}
\bibliographystyle{spmpsci}

\small{ 
\noindent  \textsc{Andrew W.\ Sale} \rule{0mm}{6mm} \\
IRMAR,
Universit\'{e} de Rennes 1, 35042 Rennes Cedex, France \\ \texttt{andrew.sale@some.oxon.org}, \
{http://perso.univ-rennes1.fr/andrew.sale/}

}

\end{document}